\documentclass[12pt, reqno]{amsart}
\allowdisplaybreaks[3]
\usepackage[left=3cm,right=3cm,top=3cm,bottom=3cm]{geometry}

\usepackage{amsthm}
\usepackage{amscd}
\usepackage{amsmath}
\usepackage{amsfonts}
\usepackage{amssymb}
\usepackage{graphicx}
\usepackage{amsopn}
\usepackage[usenames]{color}
\usepackage{tikz}
\usetikzlibrary{arrows,shapes,snakes,automata,backgrounds,petri}
\usepackage{textcomp}
\usepackage[all]{xy}

\usepackage[utf8]{inputenc}
\usepackage[T2A]{fontenc}
\usepackage{amsmath,amssymb,amsthm}
\usepackage[english]{babel}

\usepackage{hyperref}

\theoremstyle{theorem}
\newtheorem{theorem}{Theorem}
\newtheorem{corollary}[theorem]{Corollary}
\newtheorem{prop}[theorem]{Proposition}
\newtheorem{lemma}[theorem]{Lemma}
\newtheorem{conjecture}[theorem]{Conjecture}

\theoremstyle{definition}

\newtheorem{fact}[theorem]{Fact}

\newtheorem{question}[theorem]{Question}

\def\S{\Sigma}
\def\Mod{{\rm Mod}}

\def\Out{{\rm Out}}

\def\span{{\rm span}}

\def\id{{\rm id}}
\def\Z{\mathbb{Z}}
\def\Q{\mathbb{Q}}

\def\H{{\rm H}}
\def\Sp{{\rm Sp}}

\def\Homeo{{\rm Homeo}}

\numberwithin{theorem}{section}
\allowdisplaybreaks[3]

\begin{document}

\title{On the mapping class group action on the homology of surface covers}

\address{Max Planck Institute for Mathematics, Bonn 53111, Germany}
\email{spiridonovia@ya.ru}
\author{Igor Spiridonov}
\date{June 3, 2024}


\maketitle

\begin{abstract}
	Let $\phi \in {\rm Mod}(\Sigma)$ be an arbitrary element of the mapping class group of a closed orientable surface $\Sigma$ of genus at least $2$. For any characteristic cover $\widetilde{\Sigma} \to \Sigma$ one can consider the linear subspace ${\rm H}_1^{f.o.}(\widetilde{\Sigma}, \mathbb{Q})^\phi \subseteq {\rm H}_1(\widetilde{\Sigma}, \mathbb{Q})$ consisting of all homology classes with finite $\phi$-orbit. We prove that $\dim {\rm H}_1^{f.o.}(\widetilde{\S}, \mathbb{Q})^\phi$ can be arbitrary large for any fixed $\phi \in {\rm Mod}(\Sigma)$.
\end{abstract}

\section{Introduction}

\subsection{Result}

Let $\S$ be a compact orientable surface of genus at least $2$ without boundary. Let $\Mod(\S)$ be the \textit{mapping class group} of $\S$, defined by $\Mod(\S) = \pi_{0}(\Homeo^{+}(\S))$, where $\Homeo^{+}(\S)$ is the group of orientation-preserving homeomorphisms of $\S$. Let $p: \widetilde{\S} \to \S$ be a finite characteristic cover, i.e. $p_*(\pi_1(\widetilde{\S})) \lhd \pi_1(\S)$ is a finite index characteristic subgroup. Then the action of $\Mod(\S)$ on $\H_1(\widetilde{\S}, \Q)$ is well-defined up to the action of the finite deck group $\pi_1(\S) / p_*(\pi_1(\widetilde{\S}))$. Therefore, for each  $v \in \H_1(\widetilde{\S}, \Q)$ one can ask about finiteness of its orbit with respect to the action of $\Mod(\S)$ and its subgroups.

For each $\phi \in \Mod(\S)$ we consider the linear subspace $\H_1^{f.o.}(\widetilde{\S}, \Q)^\phi \subseteq \H_1(\widetilde{\S}, \Q)$ consisting of all homology classes with finite $\phi$-orbit.
The main result of this paper is as follows.

\begin{theorem} \label{mainth}
	Let $\phi \in \Mod(\S)$ be any element. Then there exist finite  characteristic covers $\widetilde{\S} \to \S$ with arbitrary large $\dim \H_1^{f.o.}(\widetilde{\S}, \Q)^\phi$.
\end{theorem}

In view of Theorem \ref{mainth}, it is natural to ask the following question, to which the author does not know the answer.

\begin{question} \label{Q}
	Let $\phi \in \Mod(\S)$ be any element and let
	$$ \dots \to \widetilde{\S}_k \to \widetilde{\S}_{k-1} \to \dots \to \widetilde{\S}_1 \to \S $$
	be any infinite tower of finite characteristic covers of $\S$ such that $\cap_{k=1}^\infty \pi_1(\widetilde{\S}_k) = \{1\}$. Is it true that $\lim_{k \to \infty} \dim \H_1^{f.o}(\widetilde{\S}_k, \Q)^\phi = \infty$?
\end{question}

Note that the usual space of $\phi$-invariants $\H_1(\widetilde{\S}, \Q)^\phi \subseteq \H_1(\widetilde{\S}, \Q)$ is well-defined after fixing a base point $x_0 \in \widetilde{\S}$. Indeed, there is a well-defined action of $\Mod(\S) \cong \Out^+(\pi_1(\S, p(x_0)))$ on the characteristic subgroup $p(\pi_1(\widetilde{\S}, x_0)) \lhd \pi_1(\S, p(x_0))$, and we have a canonical isomorphism $\H_1(\pi_1(\widetilde{\S}, x_0)), \Q) \cong \H_1(\widetilde{\S}, \Q)$. The author does not know if the statement of Theorem \ref{mainth} remains true after replacing $\dim \H_1^{f.o.}(\widetilde{\S}, \Q)^\phi$ by $\dim \H_1(\widetilde{\S}, \Q)^\phi$ in the above setting.

\subsection{History}

The (virtual) representations of the group $\Mod(\S)$ in the symplectic group $\Sp(\H_1(\widetilde{\S}, \Q))$ for regular coves $\widetilde{\S} \to \S$ are called \textit{Prym representations}. This notion is a natural generalisation of the standard symplectic representation of  $\Mod(\S)$ in $\Sp(\H_1(\S, \Q))$ and was initially introduced by Looijenga \cite{Looijenga} for the case of abelian covers. In \cite{Grunewald}, Grunewald, Larsen, Lubotzky, and Malestein constructed a lot of arithmetic quotients of the mapping class group using Prym representations.

In 2013, Putman and Wieland \cite{PutmanWieland} stated the following conjecture, which is one of the central open questions about Prym representations (the similar question for graph covers was studied by Farb and Hensel \cite{FarbHensel2}).

\begin{conjecture}\cite[Conjecture 1.2]{PutmanWieland}  \label{PWconj}
	Let $\widetilde{\S} \to \S$ be a finite characteristic cover. Then for each nonzero homology class  $v \in \H_1(\widetilde{\S}, \Q)$ the $\Mod(\S)$-orbit of $v$ is infinite.
\end{conjecture}

It turns out \cite[Theorem C]{PutmanWieland} that Conjecture \ref{PWconj} is essentially equivalent to the classical Ivanov's conjecture \cite{Ivanov} on vanishing of the first virtual Betti number of the mapping class group (see also \cite[Problem 2.11.A]{Kirby}).
It is also important to note that it is an open question whether $\Mod(\S)$ has Kazhdan's property (T), which is a stronger condition than vanishing of the first virtual Betti number.
In this context, Theorem \ref{mainth} claims that the action of any single element $\phi \in \Mod(\S)$ is not enough to prove Conjecture \ref{PWconj}.

Besides general questions on Prym representations, it is natural to study $\phi$-action on $\H_1(\widetilde{\S}, \Q)$ for a fixed element $\phi \in \Mod(\S)$. Koberda \cite{Koberda, KoberdaPhD} showed that for every nontrivial $\phi \in \Mod(\S)$ there exists a finite characteristic cover $\widetilde{\S} \to \S$ such that each lift of $\phi$ acts nontrivially on $\H_1(\widetilde{\S}, \Q)$. Liu \cite{Liu} proved the well-known McMullen's conjecture claiming that every pseudo-Anosov $\phi \in \Mod(\S)$ lifts to some finite cover $\widetilde{\S} \to \S$ with the spectral radius of the induced action on $\H_1(\widetilde{\S}, \Q)$ strictly grater than one (see also \cite{Hadari15}). Using a completely different approach, Hadari \cite{Hadari20} proved the similar result for the surfaces with at least one boundary component.
These results imply that for any infinite order element $\phi \in \Mod(\S)$ there exists a finite cover $\widetilde{\S} \to \S$ such that $\phi$ lifts to a mapping class acting on $\H_1(\widetilde{\S}, \Q)$ with infinite order. In some sense, Theorem \ref{mainth} answers the opposite question to this.

One more important aspect of Prym representations is \textit{simple closed curve homology} $\H_1^{scc}(\widetilde{\S}, \Q) \subseteq \H_1(\widetilde{\S}, \Q)$ defined as the span of the homology classes of all loops on $\widetilde{\S}$ projecting to multiples of simple curves on $\S$; this definition was initially mentioned by Farb and Hensel \cite{FarbHensel}. There exist branched (Malestein and Putman, \cite{Malestein}) and unbranched (Klukowski, \cite{Klukowski}) covers $\widetilde{\S} \to \S$ with $\H_1^{scc}(\widetilde{\S}, \Q) \neq \H_1(\widetilde{\S}, \Q)$. Recently, Boggi, Putman and Salter \cite{Boggi} proved that $\H_1^{scc}(\widetilde{\S}, \Q)$ is a symplectic subspace of $\H_1(\widetilde{\S}, \Q)$. Moreover, they showed that the analogue of Conjecture \ref{PWconj} is true for $\H_1^{scc}(\widetilde{\S}, \Q)$. Namely, for any characteristic cover $\widetilde{\S} \to \S$ and for each nonzero homology class  $v \in \H_1^{scc}(\widetilde{\S}, \Q)$ the $\Mod(\S)$-orbit of $v$ is infinite.

In view of \cite{Boggi} and Theorem \ref{mainth}, it is natural to ask if $\dim (\H_1^{f.o.}(\widetilde{\S}, \Q)^\phi  \cap \H_1^{scc}(\widetilde{\S}, \Q))$ can be arbitrary large for a fixed $\phi \in \Mod(\S)$. For periodic $\phi$ this claim is straightforward, and for reducible $\phi$ this more or less immediately follows from our proof (see Lemmas \ref{redN} and \ref{redS}). However, the author does not know if it is true for pseudo-Anosov elements $\phi \in \Mod(\S)$.

\subsection{Outline of the paper}

Our proof of Theorem \ref{mainth} uses the Nielsen-Thurston classification. We consider periodic, reducible and pseudo-Anosov elements $\phi$ separately. If $\phi$ is periodic then $\H_1^{f.o.}(\widetilde{\S}, \Q)^\phi = \H_1(\widetilde{\S}, \Q)$ and there is nothing to prove.
Reducible and pseudo-Anosov classes are considered in Sections \ref{Sec3} and \ref{Sec4}, respectively, after some preliminaries given in Section \ref{Sec2}. For reducible $\phi$ we explicitly construct covers $\widetilde{\S} \to \S$ with large  $\dim \H_1^{f.o.}(\widetilde{\S}, \Q)^\phi$. The proof in the pseudo-Anosov case is based on the deep theorem of Agol \cite{Agol13} on virtual Betti numbers of hyperbolic $3$-manifolds.

\subsection{Acknowledgements}
The author is grateful to his supervisor Ursula~Hamenstädt for useful discussions, constant attention to this work, and important suggestions on improving the paper. The author would like to thank A. Ng and C. Rudd for fruitful conversations one of which inspired him to consider the central question addressed in this work.

\section{Preliminaries} \label{Sec2}

First let us recall the Nielsen-Thurston classification \cite{Thurston} (see also \cite[Theorem 13.2]{Primer}) of the mapping class group elements. 
Let $\phi \in \Mod(\S)$ be an arbitrary element. Then precisely one of the following statements holds.
	
    (1) $\phi$ is periodic, i.e. $\phi^n = \id$ for some integer $n \geq 1$.
	
	(2) $\phi$ is reducible, i.e. $\phi$ has infinite order and $\phi^n(\alpha) = \alpha$ for some integer $n \geq 1$ and some essential simple closed curve $\alpha$ on $\S$.
	
	(3) $\phi$ is pseudo-Anosov, see  \cite[Section 13.2.3]{Primer} for the standard definition.

For any $\phi \in \Mod(\S)$ we can define the mapping torus  $M_\phi$ of $\phi$, which is a closed orientable $3$-manifold. Then we have the following classical result of Thurston \cite{Thurston}, which we will use instead of the standard definition of pseudo-Anosov elements.

\begin{theorem} \cite[Theorem 13.4]{Primer} \label{hyp}
	Let $\phi \in \Mod(\S)$ be any element. Then $M_\phi$ admits a hyperbolic metric if and only if $\phi$ is pseudo-Anosov.
\end{theorem}

In this section we also prove some facts about the action of the mapping class group on surface covers. A first step is to define $\H_1^{f.o.}(\S, \Q)^{\phi}$ for all finite regular covers, not only for characteristic ones.

Let $\widetilde{\S} \to \S$ be a finite regular cover and let $K = \pi_1(\widetilde{\S}) \lhd \pi_1(\S)$ be its fundamental group. Then an element $\phi \in \Mod(\S)$ lifts to $\Mod(\widetilde{\S})$ if and only if $K$ is invariant under the action of $\phi \in  \Mod(\S) \cong \Out(\pi_1(\S))$. However, for every $\phi \in \Mod(\S)$ there exists $k > 0$ such that $K$ is $\phi^k$-invariant, so the action of $\phi^k$ on $\H_1(K, \Q)$ is well defined. Therefore, one can consider the linear subspace $\H_1^{f.o.}(K, \Q)^\phi \subseteq \H_1(K, \Q)$ consisting of all homology classes with finite $\phi^k$-orbit. 

In this case $\phi^k$ also has a lift to $\Mod(\widetilde{\S})$ defined up to the action of the deck group $\pi_1(\S) / K$. Therefore, the action of $\phi^k$ on $\H_1(\widetilde{\S}, \Q)$ is also defined up to the action of the finite deck group, and we similarly define the linear subspace $\H_1^{f.o.}(\widetilde{\S}, \Q)^\phi \subseteq \H_1(\widetilde{\S}, \Q)$ consisting of all homology classes with finite $\phi^k$-orbit. This definition of $\H_1^{f.o.}(\widetilde{\S}, \Q)^{\phi}$ agrees with the one for characteristic covers.
We  have $\H_1^{f.o.}(K, \Q)^\phi  \cong \H_1^{f.o.}(\widetilde{\S}, \Q)^\phi$ (not canonically).

Note that these definitions do not depend on the choice of $k$ since the property of orbit (in)finiteness is preserved under passing to finite index subgroups. By the same reason we have $\H_1^{f.o.}(\widetilde{\S}, \Q)^{\phi} = \H_1^{f.o.}(\widetilde{\S}, \Q)^{\phi^n}$ for any finite regular cover $\widetilde{\S} \to \S$ and any integer $n \neq 0$.

The next lemma claims that the dimension of $\H_1^{f.o.}(K, \Q)^\phi$ can not be  decreased via passing to finite index subgroups. A related statement in a different setting appears in \cite[Lemma 2.3]{Krop}.

\begin{lemma} \label{subgr}
		Let $K', K \lhd \pi_1(\S)$ be finite index normal subgroups such that $K' \lhd K$. Then $\dim \H_1^{f.o.}(K', \Q)^\phi \geq \dim \H_1^{f.o.}(K, \Q)^\phi$.
\end{lemma}

For a group $G$ acting on a vector space $V$, we denote by $V^G$ and $V_G$ the spaces of invariants and coinvariants, respectively.
Lemma \ref{subgr} follows from the following general algebraic fact.

\begin{fact} \label{fact}
	Let $G$ be a discrete group and let $V, W$ be finite dimensional vector spaces over $\Q$ equipped with a $G$-action. Then for any surjective $G$-equivariant map $f: V \twoheadrightarrow W$ its restriction to $G$-invariants $f|_{V^G} : V^G \to W^G$ is also surjective.
\end{fact}

\begin{proof}
	We have the canonical isomorphisms $V^G \cong V_G$ and $W^G \cong W_G$ induced by the inclusions  $V^G \hookrightarrow V$ and $W^G \hookrightarrow W$. The induced map on coinvariants  $f': V_G \twoheadrightarrow W_G$ is surjective since $f$ is surjective.
	Therefore, $f|_{V^G}$ can be considered as a composite map
	$$f|_{V^G} : V^G \cong V_G \twoheadrightarrow W_G \cong W^G,$$
	which is surjective.
\end{proof}

\begin{proof}[Proof of Lemma \ref{subgr}]
	First let us choose an integer $k > 0$ such that both $K$ and $K'$ are $\phi^k$-invariant. Then the action of $\phi^k$ on $\H_1(K, \Q)$ is well-defined, so one can consider the subspaces $\H_1(K, \Q)^{\phi^{kn}} \subseteq \H_1(K, \Q)$ of $\phi^{kn}$-invariants for $n \in \mathbb{Z}$.
	Since $\H_1^{f.o.}(K, \Q)^\phi$ is finite dimensional, it follows that there exists an integer $n > 1$ such that 
	$\H_1^{f.o.}(K, \Q)^\phi = \H_1(K, \Q)^{\phi^{kn}}$. Indeed, it suffices to take $n$ equal to the least common multiple of the $\phi^k$-orbit cardinalities of basis elements of $\H_1^{f.o.}(K, \Q)^\phi$ (for some basis).
	
	Let $\iota: K' \hookrightarrow K$ be the inclusion map. Since $\iota$ comes from a finite covering of surfaces, then the induced map $\iota_*: \H_1(K', \Q) \twoheadrightarrow \H_1(K, \Q)$ is surjective. Hence Fact \ref{fact} applied to the map $\iota_*$ and the group $\Z$ acting by powers of $\phi^{kn}$ implies that we have the surjection $\H_1(K', \Q)^{\phi^{kn}} \twoheadrightarrow \H_1(K, \Q)^{\phi^{kn}}$. Since
	$$\H_1(K', \Q)^{\phi^{kn}} \subseteq \H_1^{f.o.}(K', \Q)^\phi,$$
	we obtain
	$$\dim \H_1^{f.o.}(K', \Q)^\phi \geq \dim \H_1(K', \Q)^{\phi^{kn}} \geq  \dim \H_1(K, \Q)^{\phi^{kn}} = \dim \H_1^{f.o.}(K, \Q)^\phi.$$
	This concludes the proof.
\end{proof}

We obtain the following corollary, which allows us to replace characteristic covers in the statement of Theorem \ref{mainth} by regular ones.

\begin{corollary} \label{cor1}
	Let $K \lhd \pi_1(\S)$ be a finite index normal subgroup. Then there exists a finite index characteristic subgroup $K' \lhd \pi_1(\S)$ such that $K' \lhd K$ and $\dim \H_1^{f.o.}(K', \Q)^\phi \geq \dim \H_1^{f.o.}(K, \Q)^\phi$.
\end{corollary}

\begin{proof}
	Since $K$ has finite index in $\pi_1(\S)$, then there exists a finite index characteristic subgroup $K' \lhd \pi_1(\S)$ such that $K' \lhd K$. Then by Lemma \ref{subgr} we have $\dim \H_1^{f.o.}(K', \Q)^\phi \geq \dim \H_1^{f.o.}(K, \Q)^\phi$.
\end{proof}

\section{Reducible case} \label{Sec3}

In this section we prove Theorem \ref{mainth} for a reducible $\phi$. In this case we have $\phi^k(\alpha) = \alpha$ for some integer $k > 0$ and some essential simple closed curve $\alpha$ on $\S$. We consider the cases of nonseparating and separating curves $\alpha$ separately. In both cases we explicitly construct regular covers $\widetilde{\S} \to \S$ with $\dim \H_1^{f.o.}(\widetilde{\S}, \Q)^\phi$ bounded from below. 

By a simple closed curve we mean a curve without orientation unless explicitly stated otherwise.
For an oriented simple closed curve $\alpha$ on $\S$ we denote by $[\alpha] \in \H_1(\S, \Q)$ its homology class. First we need the following result.
 
\begin{lemma} \label{lem}
	Let $\phi \in \Mod(\S)$ and let $\alpha$ be a simple closed curve such that $\phi(\alpha) = \alpha$. Let $p: \widetilde{\S} \to \S$ be a finite regular cover. Denote by $\widetilde{\alpha}_1, \dots, \widetilde{\alpha}_d$ all lifts of (multiples of) $\alpha$ to simple closed curves on $\widetilde{\S}$. Then
	$$\dim \H_1^{f.o.}(\widetilde{\S}, \Q)^\phi \geq \dim  \span \langle [\widetilde{\alpha}_1], \dots, [\widetilde{\alpha}_d] \rangle_\Q $$
	after choosing arbitrary orientations on $\widetilde{\alpha}_1, \dots, \widetilde{\alpha}_d$, where the span on the right hand side is taken inside $\H_1(\widetilde{\S}, \Q)$.
\end{lemma}

\begin{proof}
	Let us show that
	\begin{equation} \label{eq}
	\span \langle [\widetilde{\alpha}_1], \dots, [\widetilde{\alpha}_d] \rangle_\Q \subseteq  \H_1^{f.o.}(\widetilde{\S}, \Q)^\phi.
	\end{equation}
	Take an integer $k$ such that $\phi^k$ lifts to $\Mod(\widetilde{\S})$. For each $i = 1, \dots, d$ we have
	$$p(\phi^k(\widetilde{\alpha_i})) = \phi^k(p(\widetilde{\alpha_i})) = \phi^k(\alpha) = \alpha,$$
	so we obtain $\phi^k(\widetilde{\alpha_i}) \in \{\widetilde{\alpha}_1, \dots, \widetilde{\alpha}_d\}$ and hence $\phi^k([\widetilde{\alpha_i}]) \in \{\pm [\widetilde{\alpha}_1], \dots,\pm [\widetilde{\alpha}_d]\}$. Therefore, the $\phi^k$-orbit of $[\widetilde{\alpha}_i]$ is finite. Since the space $\span \langle [\widetilde{\alpha}_1], \dots, [\widetilde{\alpha}_d] \rangle_\Q$ has a basis consisting of the the homology classes $[\widetilde{\alpha}_i]$ with finite $\phi$-orbits, then (\ref{eq}) follows. This immediately implies the claim of the lemma.
\end{proof}

Let us consider the case where some nonzero power of $\phi \in \Mod(\S)$ stabilizes a nonseparating curve $\alpha$ on $\S$. Since $H_1^{f.o.}(\widetilde{\S}, \Q)^{\phi} = \H_1^{f.o.}(\widetilde{\S}, \Q)^{\phi^k}$ for any integer $k \neq 0$, we can replace $\phi$ by $\phi^k$ and assume that $\phi(\alpha) = \alpha$.

\begin{figure}[h]
	\scalebox{1}{
		\begin{tikzpicture}
			
			
			
			\draw[dotted, very thick] (-2 + 0.71, 2 + 0.71) arc (90+22.5:135+22.5:2.63);

			\node[] at (0, 1.41*5.3+0.5) {$2\pi/N$};
			
			\draw[] (5.3, 5.3) arc (45:135:1.41*5.3);
			\draw[] (5.3, 5.3) to (5.3-0.5, 5.3);
			\draw[] (5.3, 5.3) to (5.3, 5.3+0.5);
			
			\draw[] (1, 2.83) to (1, 3 + 2.83);
			\draw[] (-1, 2.83) to (-1, 3 + 2.83);
			\draw[] (1, 3 + 2.83) arc (0:180:1);

			\draw[] (2.83, 1) to (2.83 + 3, 1);
			\draw[] (2.83, -1) to (2.83 + 3, -1);
			\draw[] (2.83 + 3, -1) arc (-90:90:1);
			
			\draw[] (-2.83, 1) to (-2.83 - 3, 1);
			\draw[] (-2.83, -1) to (-2.83 - 3, -1);
			\draw[] (-2.83 - 3, 1) arc (90:270:1);
			
			\draw[] (-1, -2.83) to (-1, -3 - 2.83);
			\draw[] (1, -2.83) to (1, -3 - 2.83);
			\draw[] (-1, -3 - 2.83) arc (180:360:1);
			
			\draw[] (2 - 0.71, 2 + 0.71) to (2 - 0.71 + 3 * 0.71, 2 + 0.71 + 3*0.71);
			\draw[] (2 + 0.71, 2 - 0.71) to (2 + 0.71 + 3 * 0.71, 2 - 0.71 + 3*0.71);
			\draw[] (2 - 0.71 + 3 * 0.71, 2 + 0.71 + 3*0.71) arc (135:135-180:1.005);
			
			\draw[] (-2 + 0.71, -2 - 0.71) to (-2 + 0.71 - 3 * 0.71, -2 - 0.71 - 3*0.71);
			\draw[] (-2 - 0.71, -2 + 0.71) to (-2 - 0.71 - 3 * 0.71, -2 + 0.71 - 3*0.71);
			\draw[] (-2 - 0.71 - 3 * 0.71, -2 + 0.71 - 3*0.71) arc (135:135+180:1.005);
			
			\draw[] (2 - 0.71, -2 - 0.71) to (2 - 0.71 + 3 * 0.71, -2 - 0.71 - 3*0.71);
			\draw[] (2 + 0.71, -2 + 0.71) to (2 + 0.71 + 3 * 0.71, -2 + 0.71 - 3*0.71);
			\draw[] (2 + 0.71 + 3 * 0.71, -2 + 0.71 - 3*0.71) arc (45:45-180:1.005);

			\draw[] (1, 2.83) to [out=-90, in=225]  (2 - 0.71, 2 + 0.71);
			\draw[] (-1, 2.83) to [out=-90, in=-45]  (-2 + 0.71, 2 + 0.71);
			\draw[] (2.83, 1) to [out=180, in=225]  (2 + 0.71, 2 - 0.71);
			\draw[] (2.83, -1) to [out=180, in=135]  (2 + 0.71, -2 + 0.71);
			
			\draw[] (-1, -2.83) to [out=-90+180, in=225+180]  (-2 + 0.71, -2 - 0.71);
			\draw[] (1, -2.83) to [out=-90+180, in=-45+180]  (2 - 0.71, -2 - 0.71);
			\draw[] (-2.83, -1) to [out=0, in=225+180]  (-2 - 0.71, -2 + 0.71);
			\draw[] (-2.83, 1) to [out=0, in=135+180]  (-2 - 0.71, 2 - 0.71);
			
			\draw[] (0, 0) circle (2);
			
			\draw[] (0, 2.83+0.5) circle (0.3);
			\draw[] (0, 2.83 + 1.5) circle (0.3);
			\draw[] (0, 2.83 + 3.5) circle (0.3);
			\draw[red] (0, 2.83 + 3.5) circle (0.4);
			\draw[dotted, very thick] (0, 2.83 + 2.2) to (0, 2.83 + 2.8);
			
			\node[red, scale = 0.9] at (-0.4, 2.4 + 3.45) {$\widetilde{\alpha}_{N-1}$};
			
			\draw[] (0, -2.83 - 0.5) circle (0.3);
			\draw[] (0, -2.83 - 1.5) circle (0.3);
			\draw[] (0, -2.83 - 3.5) circle (0.3);
			\draw[red] (0, -2.83 - 3.5) circle (0.4);
			\draw[dotted, very thick] (-0, -2.83 - 2.2) to (0, -2.83 - 2.8);
			
			\node[red, scale = 0.9] at (-0.6, -2.83 - 3.5 + 0.3) {$\widetilde{\alpha}_{3}$};

			\draw[] (-2.83 - 0.5, 0) circle (0.3);
			\draw[] (-2.83 - 1.5, 0) circle (0.3);
			\draw[] (-2.83 - 3.5, 0) circle (0.3);
			\draw[red] (-2.83 - 3.5, 0) circle (0.4);
			\draw[dotted, very thick] (-2.83 - 2.2, 0) to (-2.83 - 2.8, 0);
			\node[red, scale = 0.9] at (-2.83 - 3.5 + 0.4, 0.5) {$\widetilde{\alpha}_{5}$};

			\draw[] (2.83 + 0.5, 0) circle (0.3);
			\draw[] (2.83 + 1.5, 0) circle (0.3);
			\draw[] (2.83 + 3.5, 0) circle (0.3);
			\draw[red] (2.83 + 3.5, 0) circle (0.4);
			\draw[dotted, very thick] (2.83 + 2.2, 0) to (2.83 + 2.8, 0);
			\node[red, scale = 0.9] at (2.83 + 3.5 - 0.3, 0.6) {$\widetilde{\alpha}_{1}$};
			
			\draw[] (0.71*2.83 + 0.71*0.5, 0.71*2.83 + 0.71*0.5) circle (0.3);
			\draw[] (0.71*2.83 + 0.71*1.5, 0.71*2.83 + 0.71*1.5) circle (0.3);
			\draw[] (0.71*2.83 + 0.71*3.5, 0.71*2.83 + 0.71*3.5) circle (0.3);
			\draw[red] (0.71*2.83 + 0.71*3.5, 0.71*2.83 + 0.71*3.5) circle (0.4);
			\draw[dotted, very thick] (0.71*2.83 + 0.71*2.2, 0.71*2.83 + 0.71*2.2) to (0.71*2.83 + 0.71*2.8, 0.71*2.83 + 0.71*2.8);
			\node[red, scale = 0.9] at (0.71*2.83 + 0.71*3.5, 0.71*2.83 + 0.71*3.5 - 0.7) {$\widetilde{\alpha}_{N}$};
			
			\draw[] (0.71*2.83 + 0.71*0.5, -0.71*2.83 - 0.71*0.5) circle (0.3);
			\draw[] (0.71*2.83 + 0.71*1.5, -0.71*2.83 - 0.71*1.5) circle (0.3);
			\draw[] (0.71*2.83 + 0.71*3.5, -0.71*2.83 - 0.71*3.5) circle (0.3);
			\draw[red] (0.71*2.83 + 0.71*3.5, -0.71*2.83 - 0.71*3.5) circle (0.4);
			\draw[dotted, very thick] (0.71*2.83 + 0.71*2.2, -0.71*2.83 - 0.71*2.2) to (0.71*2.83 + 0.71*2.8, -0.71*2.83 - 0.71*2.8);
			\node[red, scale = 0.9] at (0.71*2.83 + 0.71*3.5 - 0.6, -0.71*2.83 - 0.71*3.5 - 0.2) {$\widetilde{\alpha}_{2}$};

			\draw[] (-0.71*2.83 - 0.71*0.5, -0.71*2.83 - 0.71*0.5) circle (0.3);
			\draw[] (-0.71*2.83 - 0.71*1.5, -0.71*2.83 - 0.71*1.5) circle (0.3);
			\draw[] (-0.71*2.83 - 0.71*3.5, -0.71*2.83 - 0.71*3.5) circle (0.3);
			\draw[red] (-0.71*2.83 - 0.71*3.5, -0.71*2.83 - 0.71*3.5) circle (0.4);
			\draw[dotted, very thick] (-0.71*2.83 - 0.71*2.2, -0.71*2.83 - 0.71*2.2) to (-0.71*2.83 - 0.71*2.8, -0.71*2.83 - 0.71*2.8);
			\node[red, scale = 0.9] at (-0.71*2.83 - 0.71*3.5 + 0.7, -0.71*2.83 - 0.71*3.5 - 0.2) {$\widetilde{\alpha}_{4}$};
			
	\end{tikzpicture}}
	\caption{The $(\Z/N\Z)$-action on $\widetilde{\S}$ and the curves $\widetilde{\alpha}_1, \dots, \widetilde{\alpha}_N$.}
	\label{S-curve_1}
\end{figure}

\begin{lemma} \label{redN}
	Let $\phi \in \Mod(\S)$ be such that there exists a nonseparating simple closed curve $\alpha$ with $\phi(\alpha) = \alpha$ and let $N > 0$ be an integer.
	Then there exists a regular finite cover $\widetilde{\S} \to \S$ with $\dim \H_1^{f.o.}(\widetilde{\S}, \Q)^\phi \geq N$.
\end{lemma}

\begin{proof}
	Let $g \geq 2$ be the genus of $\S$. Consider the surface $\widetilde{\S}$ of genus $1+N(g-1)$ shown in Fig. \ref{S-curve_1}, consisting of the connected sum of $N$ subsurfaces of genus $g-1$ and the torus. The group $\Z / N\Z$ acts on $\widetilde{\S}$ by rotations and we obtain a finite regular covering $p: \widetilde{\S} \to \S' = \widetilde{\S} / (\Z / N\Z)$. 
	
	Let us fix $N$ curves $\widetilde{\alpha}_1, \dots, \widetilde{\alpha}_N$ on $\widetilde{\S}$ as shown in Fig. \ref{S-curve_1}. Note that all these curves belong to the same $(\Z / N\Z)$-orbit, let us choose orientations on these curves agreed with this action.
	Since the surface $\widetilde{\S} \setminus \cup \{ \widetilde{\alpha}_1, \dots, \widetilde{\alpha}_N \}$ is connected, it follows that the homology classes $[\widetilde{\alpha}_1], \dots, [\widetilde{\alpha}_N] \in \H_1(\S, \Q)$ are linearly independent. Hence we obtain $$\dim  \span \langle [\widetilde{\alpha}_1], \dots, [\widetilde{\alpha}_N] \rangle_\Q = N.$$
	
	Denote $\alpha' = p(\widetilde{\alpha}_1)$. Note that $\alpha'$ is a nonseparating simple closed curve on $\S'$ and $p(\widetilde{\alpha}_i) = \alpha'$ for all $i = 1, \dots, N$.
	Since all nonseparating simple closed curves on $\S$ are $\Mod(\S)$-equivalent, one can choose a homeomorphism $\eta : \S' \cong \S$ such that $\eta(\alpha') = \alpha$. We obtain the covering  $\eta \circ p : \widetilde{\S} \to \S$ such that $\widetilde{\alpha}_1, \dots, \widetilde{\alpha}_N$ are the lifts of $\alpha$. By Lemma \ref{lem} we have
	$$\dim \H_1^{f.o.}(\widetilde{\S}, \Q)^\phi \geq \dim  \span \langle [\widetilde{\alpha}_1], \dots, [\widetilde{\alpha}_N] \rangle_\Q = N,$$
	 which concludes the proof.
\end{proof}

\begin{figure}[h]
	\scalebox{1}{
		\begin{tikzpicture}
			
			
			
			
			\draw[dotted, very thick] (-2 + 0.71, 2 + 0.71) arc (90+22.5:135+22.5:2.63);
			
			\draw[] (6, 6) arc (45:135:1.41*6);
			\draw[] (6, 6) to (6-0.5, 6);
			\draw[] (6, 6) to (6, 6+0.5);
			\node[] at (0, 1.41*6+0.5) {$2\pi/(N+1)$};

			\draw[] (1, 2.83) to (1, 3 + 2.83);
			\draw[] (-1, 2.83) to (-1, 3 + 2.83);

			\node[red, scale=0.9] at (1-0.35, 3 + 2.83+0.4) {$\widetilde{\delta}^{1}_{N}$};
			\node[red, scale=0.9] at (1-0.35, 2 + 2.83+0.4) {$\widetilde{\delta}^{2}_{N}$};
			\node[red, scale=0.9] at (1-0.35, 1 + 2.83+0.4) {$\widetilde{\delta}^{g-2}_{N}$};
			\node[red, scale=0.9] at (1-0.35, 0 + 2.83+0.4) {$\widetilde{\delta}^{g-1}_{N}$};

			\node[red, scale=0.9] at (-1+0.35, -3 - 2.83-0.4) {$\widetilde{\delta}^{1}_{3}$};
			\node[red, scale=0.9] at (-1+0.35, -2 - 2.83-0.4) {$\widetilde{\delta}^{2}_{3}$};
			\node[red, scale=0.9] at (-1+0.35, -1 - 2.83-0.4) {$\widetilde{\delta}^{g-2}_{3}$};
			\node[red, scale=0.9] at (-1+0.35, -0 - 2.83-0.4) {$\widetilde{\delta}^{g-1}_{3}$};

			\node[red, scale=0.9] at (3 + 2.83+0.35, -1+0.35) {$\widetilde{\delta}^{1}_{1}$};
			\node[red, scale=0.9] at (2 + 2.83+0.35, -1+0.35) {$\widetilde{\delta}^{2}_{1}$};
			\node[red, scale=0.9] at (1 + 2.83+0.45, -1+0.35) {$\widetilde{\delta}^{g-2}_{1}$};
			\node[red, scale=0.9] at (0 + 2.83+0.45, -1+0.35) {$\widetilde{\delta}^{g-1}_{1}$};

			\node[red, scale=0.9] at (-3 - 2.83-0.3, 1-0.35) {$\widetilde{\delta}^{1}_{5}$};
			\node[red, scale=0.9] at (-2 - 2.83-0.3, 1-0.35) {$\widetilde{\delta}^{2}_{5}$};
			\node[red, scale=0.9] at (-1 - 2.83-0.4, 1-0.35) {$\widetilde{\delta}^{g-2}_{5}$};
			\node[red, scale=0.9] at (-0 - 2.83-0.4, 1-0.35) {$\widetilde{\delta}^{g-1}_{5}$};

			\draw[red] (-1, 3 + 2.83) to [out=10, in=170] (1, 3 + 2.83);
			\draw[red, dashed] (-1, 3 + 2.83) to [out=-10, in=-170] (1, 3 + 2.83);
			\draw[red] (-1, 2 + 2.83) to [out=10, in=170] (1, 2 + 2.83);
			\draw[red, dashed] (-1, 2 + 2.83) to [out=-10, in=-170] (1, 2 + 2.83);
			
			\draw[red] (-1, 1 + 2.83) to [out=10, in=170] (1, 1 + 2.83);
			\draw[red, dashed] (-1, 1 + 2.83) to [out=-10, in=-170] (1, 1 + 2.83);
			\draw[red] (-1, 2.83) to [out=10, in=170] (1, 2.83);
			\draw[red, dashed] (-1, 2.83) to [out=-10, in=-170] (1, 2.83);

			\draw[red] (-1, -3 - 2.83) to [out=-10, in=-170] (1, -3 - 2.83);
			\draw[red, dashed] (-1, -3 - 2.83) to [out=10, in=170] (1, -3 - 2.83);
			\draw[red] (-1, -2 - 2.83) to [out=-10, in=-170] (1, -2 - 2.83);
			\draw[red, dashed] (-1, -2 - 2.83) to [out=10, in=170] (1, -2 - 2.83);
			
			\draw[red] (-1, -1 - 2.83) to [out=-10, in=-170] (1, -1 - 2.83);
			\draw[red, dashed] (-1, -1 - 2.83) to [out=10, in=170] (1, -1 - 2.83);
			\draw[red] (-1, -2.83) to [out=-10, in=-170] (1, -2.83);
			\draw[red, dashed] (-1, -2.83) to [out=10, in=170] (1, -2.83);

			\draw[red] (-3 - 2.83, -1) to [out=100, in=-100] (-3 - 2.83, 1);
			\draw[red, dashed] (-3 - 2.83, -1) to [out=80, in=-80] (-3 - 2.83, 1);
			\draw[red] (-2 - 2.83, -1) to [out=100, in=-100] (-2 - 2.83, 1);
			\draw[red, dashed] (-2 - 2.83, -1) to [out=80, in=-80] (-2 - 2.83, 1);
			
			\draw[red] (-1 - 2.83, -1) to [out=100, in=-100] (-1 - 2.83, 1);
			\draw[red, dashed] (-1 - 2.83, -1) to [out=80, in=-80] (-1- 2.83, 1);
			\draw[red] (- 2.83, -1) to [out=100, in=-100] (-2.83, 1);
			\draw[red, dashed] (- 2.83, -1) to [out=80, in=-80] (- 2.83, 1);

			\draw[red, dashed] (3 + 2.83, -1) to [out=100, in=-100] (3 + 2.83, 1);
			\draw[red] (3 + 2.83, -1) to [out=80, in=-80] (3 + 2.83, 1);
			\draw[red, dashed] (2 + 2.83, -1) to [out=100, in=-100] (2 + 2.83, 1);
			\draw[red] (2 + 2.83, -1) to [out=80, in=-80] (2 + 2.83, 1);
			
			\draw[red, dashed] (1 + 2.83, -1) to [out=100, in=-100] (1 + 2.83, 1);
			\draw[red] (1 + 2.83, -1) to [out=80, in=-80] (1+2.83, 1);
			\draw[red, dashed] (2.83, -1) to [out=100, in=-100] (2.83, 1);
			\draw[red] (2.83, -1) to [out=80, in=-80] (2.83, 1);
			
			\node[red, scale=0.9] at (2 + 0.71 + 3 * 0.71 + 0.4, 2 - 0.71 + 3*0.71 - 0.3) {$\widetilde{\delta}^{1}_{N+1}$};
			\node[red, scale=0.9] at (2 + 0.71 + 2 * 0.71 + 0.4, 2 - 0.71 + 2*0.71 - 0.3) {$\widetilde{\delta}^{2}_{N+1}$};
			\node[red, scale=0.9] at (2 + 0.71 + 1 * 0.71 + 0.4, 2 - 0.71 + 1*0.71 - 0.3) {$\widetilde{\delta}^{g-2}_{N+1}$};
			\node[red, scale=0.9] at (2 + 0.71 + 0 * 0.71 -0.6, 2 - 0.71 + 0*0.71) {$\widetilde{\delta}^{g-1}_{N+1}$};
			
			\node[red, scale=0.9] at (2 + 0.71 + 3 * 0.71 + 0.3, -2 + 0.71 - 3*0.71 + 0.3) {$\widetilde{\delta}^{1}_{2}$};
			\node[red, scale=0.9] at (2 + 0.71 + 2 * 0.71 + 0.3, -2 + 0.71 - 2*0.71 + 0.3) {$\widetilde{\delta}^{2}_{2}$};
			\node[red, scale=0.9] at (2 + 0.71 + 1 * 0.71 + 0.3, -2 + 0.71 - 1*0.71 + 0.3) {$\widetilde{\delta}^{g-2}_{2}$};
			\node[red, scale=0.9] at (2 + 0.71 + 0 * 0.71 -0.6, -2 + 0.71 - 0*0.71 -0.1) {$\widetilde{\delta}^{g-1}_{2}$};
			
			\node[red, scale=0.9] at (-2 - 0.71 - 3 * 0.71 - 0.2, -2 + 0.71 - 3*0.71 + 0.3) {$\widetilde{\delta}^{1}_{4}$};
			\node[red, scale=0.9] at (-2 - 0.71 - 2 * 0.71 - 0.2, -2 + 0.71 - 2*0.71 + 0.3) {$\widetilde{\delta}^{2}_{4}$};
			\node[red, scale=0.9] at (-2 - 0.71 - 1 * 0.71 - 0.1, -2 + 0.71 - 1*0.71 + 0.3) {$\widetilde{\delta}^{g-2}_{4}$};
			\node[red, scale=0.9] at (-2 - 0.71 - 0 * 0.71 + 0.7, -2 + 0.71 - 0*0.71 -0.1) {$\widetilde{\delta}^{g-1}_{4}$};

			\draw[red] (2 - 0.71 + 3 * 0.71, 2 + 0.71 + 3*0.71) to [out=10-45, in=170-45] (2 + 0.71 + 3 * 0.71, 2 - 0.71 + 3*0.71);
			\draw[red, dashed] (2 - 0.71 + 3 * 0.71, 2 + 0.71 + 3*0.71) to [out=-10-45, in=-170-45] (2 + 0.71 + 3 * 0.71, 2 - 0.71 + 3*0.71);
			
			\draw[red] (2 - 0.71 + 2 * 0.71, 2 + 0.71 + 2*0.71) to [out=10-45, in=170-45] (2 + 0.71 + 2 * 0.71, 2 - 0.71 + 2*0.71);
			\draw[red, dashed] (2 - 0.71 + 2 * 0.71, 2 + 0.71 + 2*0.71) to [out=-10-45, in=-170-45] (2 + 0.71 + 2 * 0.71, 2 - 0.71 + 2*0.71);
			
			\draw[red] (2 - 0.71 + 1 * 0.71, 2 + 0.71 + 1*0.71) to [out=10-45, in=170-45] (2 + 0.71 + 1 * 0.71, 2 - 0.71 + 1*0.71);
			\draw[red, dashed] (2 - 0.71 + 1 * 0.71, 2 + 0.71 + 1*0.71) to [out=-10-45, in=-170-45] (2 + 0.71 + 1 * 0.71, 2 - 0.71 + 1*0.71);
			
			\draw[red] (2 - 0.71 + 0 * 0.71, 2 + 0.71 + 0*0.71) to [out=10-45, in=170-45] (2 + 0.71 + 0 * 0.71, 2 - 0.71 + 0*0.71);
			\draw[red, dashed] (2 - 0.71 + 0 * 0.71, 2 + 0.71 + 0*0.71) to [out=-10-45, in=-170-45] (2 + 0.71 + 0 * 0.71, 2 - 0.71 + 0*0.71);

			\draw[red] (-2 + 0.71 - 3 * 0.71, -2 - 0.71 - 3*0.71) to [out=10+135, in=-10-45] (-2 - 0.71 - 3 * 0.71, -2 + 0.71 - 3*0.71);
			\draw[red, dashed] (-2 + 0.71 - 3 * 0.71, -2 - 0.71 - 3*0.71) to [out=-10+135, in=10-45] (-2 - 0.71 - 3 * 0.71, -2 + 0.71 - 3*0.71);
			
			\draw[red] (-2 + 0.71 - 2* 0.71, -2 - 0.71 - 2*0.71) to [out=10+135, in=-10-45] (-2 - 0.71 - 2 * 0.71, -2 + 0.71 - 2*0.71);
			\draw[red, dashed] (-2 + 0.71 - 2 * 0.71, -2 - 0.71 - 2*0.71) to [out=-10+135, in=10-45] (-2 - 0.71 - 2 * 0.71, -2 + 0.71 - 2*0.71);
			
			\draw[red] (-2 + 0.71 - 1* 0.71, -2 - 0.71 - 1*0.71) to [out=10+135, in=-10-45] (-2 - 0.71 - 1 * 0.71, -2 + 0.71 - 1*0.71);
			\draw[red, dashed] (-2 + 0.71 - 1 * 0.71, -2 - 0.71 - 1*0.71) to [out=-10+135, in=10-45] (-2 - 0.71 - 1 * 0.71, -2 + 0.71 - 1*0.71);
			
			\draw[red] (-2 + 0.71 - 0 * 0.71, -2 - 0.71 - 0*0.71) to [out=10+135, in=-10-45] (-2 - 0.71 - 0 * 0.71, -2 + 0.71 - 0*0.71);
			\draw[red, dashed] (-2 + 0.71 - 0 * 0.71, -2 - 0.71 - 0*0.71) to [out=-10+135, in=10-45] (-2 - 0.71 - 0 * 0.71, -2 + 0.71 - 0*0.71);

			\draw[red] (2 - 0.71 + 3 * 0.71, -2 - 0.71 - 3*0.71) to [out=-10+45, in=10-135] (2 + 0.71 + 3 * 0.71, -2 + 0.71 - 3*0.71);
			\draw[red, dashed] (2 - 0.71 + 3 * 0.71, -2 - 0.71 - 3*0.71) to [out=10+45, in=-10-135] (2 + 0.71 + 3 * 0.71, -2 + 0.71 - 3*0.71);
			
			\draw[red] (2 - 0.71 + 2 * 0.71, -2 - 0.71 - 2*0.71) to [out=-10+45, in=10-135] (2 + 0.71 + 2 * 0.71, -2 + 0.71 - 2*0.71);
			\draw[red, dashed] (2 - 0.71 + 2 * 0.71, -2 - 0.71 - 2*0.71) to [out=10+45, in=-10-135] (2 + 0.71 + 2 * 0.71, -2 + 0.71 - 2*0.71);
			
			\draw[red] (2 - 0.71 + 1* 0.71, -2 - 0.71 - 1*0.71) to [out=-10+45, in=10-135] (2 + 0.71 + 1 * 0.71, -2 + 0.71 - 1*0.71);
			\draw[red, dashed] (2 - 0.71 + 1 * 0.71, -2 - 0.71 - 1*0.71) to [out=10+45, in=-10-135] (2 + 0.71 + 1 * 0.71, -2 + 0.71 - 1*0.71);
			
			\draw[red] (2 - 0.71 + 0 * 0.71, -2 - 0.71 - 0*0.71) to [out=-10+45, in=10-135] (2 + 0.71 + 0 * 0.71, -2 + 0.71 - 0*0.71);
			\draw[red, dashed] (2 - 0.71 + 0* 0.71, -2 - 0.71 - 0*0.71) to [out=10+45, in=-10-135] (2 + 0.71 + 0 * 0.71, -2 + 0.71 - 0*0.71);

			\draw[dotted, very thick] (-1, 3 + 2.83) arc (90:180: 2 + 2.83);

			\draw[] (2.83, 1) to (2.83 + 3, 1);
			\draw[] (2.83, -1) to (2.83 + 3, -1);

			\draw[] (-2.83, 1) to (-2.83 - 3, 1);
			\draw[] (-2.83, -1) to (-2.83 - 3, -1);

			\draw[] (-1, -2.83) to (-1, -3 - 2.83);
			\draw[] (1, -2.83) to (1, -3 - 2.83);

			\draw[] (2 - 0.71, 2 + 0.71) to (2 - 0.71 + 3 * 0.71, 2 + 0.71 + 3*0.71);
			\draw[] (2 + 0.71, 2 - 0.71) to (2 + 0.71 + 3 * 0.71, 2 - 0.71 + 3*0.71);

			\draw[] (-2 + 0.71, -2 - 0.71) to (-2 + 0.71 - 3 * 0.71, -2 - 0.71 - 3*0.71);
			\draw[] (-2 - 0.71, -2 + 0.71) to (-2 - 0.71 - 3 * 0.71, -2 + 0.71 - 3*0.71);

			\draw[] (2 - 0.71, -2 - 0.71) to (2 - 0.71 + 3 * 0.71, -2 - 0.71 - 3*0.71);
			\draw[] (2 + 0.71, -2 + 0.71) to (2 + 0.71 + 3 * 0.71, -2 + 0.71 - 3*0.71);

			\draw[] (1, 3 + 2.83) to [out=90, in=45]  (2 - 0.71 + 3 * 0.71, 2 + 0.71 + 3*0.71);
			\draw[] (-1, -3 - 2.83) to [out=-90, in=-135]  (-2 + 0.71 - 3 * 0.71, -2 - 0.71 - 3*0.71);
			
			\draw[] (-3 - 2.83, -1) to [out=180, in=-135]  (-2 - 0.71 - 3*0.71, -2 + 0.71 - 3 * 0.71);	
			\draw[] (3 + 2.83, -1) to [out=0, in=-45]  (+2 + 0.71 + 3*0.71, -2 + 0.71 - 3 * 0.71);	
			
			\draw[] (1, -3 - 2.83) to [out=-90, in=-45]  (2 - 0.71 + 3 * 0.71, -2 - 0.71 - 3*0.71);
			\draw[] (3 + 2.83, 1) to [out=0, in=45]  (+2 + 0.71 + 3*0.71, 2 - 0.71 + 3 * 0.71);	
			
			\draw[] (1, 2.83) to [out=-90, in=225]  (2 - 0.71, 2 + 0.71);
			\draw[] (-1, 2.83) to [out=-90, in=-45]  (-2 + 0.71, 2 + 0.71);
			\draw[] (2.83, 1) to [out=180, in=225]  (2 + 0.71, 2 - 0.71);
			\draw[] (2.83, -1) to [out=180, in=135]  (2 + 0.71, -2 + 0.71);
			
			\draw[] (-1, -2.83) to [out=-90+180, in=225+180]  (-2 + 0.71, -2 - 0.71);
			\draw[] (1, -2.83) to [out=-90+180, in=-45+180]  (2 - 0.71, -2 - 0.71);
			\draw[] (-2.83, -1) to [out=0, in=225+180]  (-2 - 0.71, -2 + 0.71);
			\draw[] (-2.83, 1) to [out=0, in=135+180]  (-2 - 0.71, 2 - 0.71);
			
			\draw[] (0, 0) circle (2);
			\draw[] (0, 0) circle (7.2);
			
			\draw[] (0, 2.83+0.5) circle (0.3);
			\draw[] (0, 2.83 + 2.5) circle (0.3);
			\draw[dotted, very thick] (0, 2.83 + 1.2) to (0, 2.83 + 1.8);

			\draw[] (0, -2.83 - 0.5) circle (0.3);
			\draw[] (0, -2.83 - 2.5) circle (0.3);
			\draw[dotted, very thick] (-0, -2.83 - 1.2) to (0, -2.83 - 1.8);

			\draw[] (-2.83 - 0.5, 0) circle (0.3);
			\draw[] (-2.83 - 2.5, 0) circle (0.3);
			\draw[dotted, very thick] (-2.83 - 1.2, 0) to (-2.83 - 1.8, 0);
			
			\draw[] (2.83 + 0.5, 0) circle (0.3);
			\draw[] (2.83 + 2.5, 0) circle (0.3);
			\draw[dotted, very thick] (2.83 + 1.2, 0) to (2.83 + 1.8, 0);
			
			\draw[] (0.71*2.83 + 0.71*0.5, 0.71*2.83 + 0.71*0.5) circle (0.3);
			\draw[] (0.71*2.83 + 0.71*2.5, 0.71*2.83 + 0.71*2.5) circle (0.3);
			\draw[dotted, very thick] (0.71*2.83 + 0.71*1.2, 0.71*2.83 + 0.71*1.2) to (0.71*2.83 + 0.71*1.8, 0.71*2.83 + 0.71*1.8);
			
			\draw[] (0.71*2.83 + 0.71*0.5, -0.71*2.83 - 0.71*0.5) circle (0.3);
			\draw[] (0.71*2.83 + 0.71*2.5, -0.71*2.83 - 0.71*2.5) circle (0.3);
			\draw[dotted, very thick] (0.71*2.83 + 0.71*1.2, -0.71*2.83 - 0.71*1.2) to (0.71*2.83 + 0.71*1.8, -0.71*2.83 - 0.71*1.8);

			\draw[] (-0.71*2.83 - 0.71*0.5, -0.71*2.83 - 0.71*0.5) circle (0.3);
			\draw[] (-0.71*2.83 - 0.71*2.5, -0.71*2.83 - 0.71*2.5) circle (0.3);
			\draw[dotted, very thick] (-0.71*2.83 - 0.71*1.2, -0.71*2.83 - 0.71*1.2) to (-0.71*2.83 - 0.71*1.8, -0.71*2.83 - 0.71*1.8);

	\end{tikzpicture}}
	\caption{The $(\Z/(N+1)\Z)$-action on $\widetilde{\S}$ and the curves $\widetilde{\delta}_i^j$.}
	\label{S-curve_2}
\end{figure}

Now we consider the case where some nonzero power of $\phi \in \Mod(\S)$ stabilizes a separating curve $\delta$ on $\S$. Similarly, we can replace $\phi$ by its power and assume that $\phi(\delta) = \delta$.

\begin{lemma} \label{redS}
	Let $\phi \in \Mod(\S)$ be such that there exists a separating simple closed curve $\delta$ with $\phi(\delta) = \delta$ and let $N > 0$ be an integer.
	Then there exists a regular finite cover $\widetilde{\S} \to \S$ with $\dim \H_1^{f.o.}(\widetilde{\S}, \Q)^\phi \geq N$.
\end{lemma}

\begin{proof}
Let $g \geq 2$ be the genus of $\S$. Consider the surface $\widetilde{\S}$ of genus $1+(N+1)(g-1)$ shown in Fig. \ref{S-curve_2}, consisting two tori connected by $N+1$ subsurfaces of genus $g-2$. 
The group $\Z / (N+1)\Z$ acts on $\widetilde{\S}$ by rotations and we obtain a finite regular covering $p: \widetilde{\S} \to \S' = \widetilde{\S} / (\Z / (N+1)\Z)$. 

For each $j = 1, \dots, g-1$ let us fix $N+1$ curves $\widetilde{\delta}^j_1, \dots, \widetilde{\delta}^j_{N+1}$ on $\widetilde{\S}$ as shown in Fig. \ref{S-curve_2}. Note that for a fixed $j$ all these curves belong to the same $(\Z / (N+1)\Z)$-orbit, let us choose orientations on these curves agreed with this action. Since the surface $\widetilde{\S} \setminus \cup \{ \widetilde{\delta}^j_1, \dots, \widetilde{\delta}^j_N \}$ is connected, it follows that the homology classes $[\widetilde{\delta}^j_1], \dots, [\widetilde{\delta}^j_N] \in \H_1(\S, \Q)$ are linearly independent for each $j$. One can easily check that $\widetilde{\S} \setminus \cup \{ \widetilde{\delta}^j_1, \dots, \widetilde{\delta}^j_{N+1} \}$ is not connected, so 
$$\dim  \span \langle [\widetilde{\delta}^j_1], \dots, [\widetilde{\delta}^j_{N+1}] \rangle_\Q = N$$
for any $j = 1, \dots, g-1$.

For each $j$ denote $\delta^j = p(\widetilde{\delta}^j_1)$. One can check that $\delta$ is a separating simple closed curve on $\S'$ bounding a subsurface of genus $j$. Moreover, we have $p(\widetilde{\delta}^j_i) = \delta^j$ for all $i = 1, \dots, N+1$. Denote by $h$ the genus of one of the subsurfaces bounded by $\delta$. Since the $\Mod(\S)$-orbit of $\delta$ is uniquely determined by $h$,  one can choose a homeomorphism $\eta : \S' \cong \S$ such that $\eta(\delta^h) = \delta$. We obtain the covering  $\eta \circ p : \widetilde{\S} \to \S$ such that $\widetilde{\delta}^h_1, \dots, \widetilde{\delta}^h_{N+1}$ are the lifts of $\delta$. By Lemma \ref{lem} we have
$$\dim \H_1^{f.o.}(\widetilde{\S}, \Q)^\phi \geq \dim  \span \langle [\widetilde{\delta}^h_1], \dots, [\widetilde{\delta}^h_{N+1}] \rangle_\Q = N,$$
which concludes the proof.
\end{proof}

\begin{prop} \label{red}
Let $\phi \in \Mod(\S)$ be a reducible element and let $N > 0$ be an integer. Then there exists a regular finite cover $\widetilde{\S} \to \S$ with $\dim \H_1^{f.o.}(\widetilde{\S}, \Q)^\phi \geq N$.
\end{prop}

\begin{proof}
	Since $\phi \in \Mod(\S)$ is reducible, it follows that there exists an integer $n > 0$ and a simple closed curve $\alpha$ on $\S$ such that $\phi^n(\alpha) = \alpha$. Since any curve is either nonseparating or separating, Lemmas \ref{redN} and \ref{redS} imply that in both cases there exists a regular finite cover $\widetilde{\S} \to \S$ with $\dim \H_1^{f.o.}(\widetilde{\S}, \Q)^{\phi^n} \geq N$.
	We obtain
	$$ \dim \H_1^{f.o.}(\widetilde{\S}, \Q)^{\phi} = \dim \H_1^{f.o.}(\widetilde{\S}, \Q)^{\phi^n} \geq N.$$
	This concludes the proof.
\end{proof}

\section{Pseudo-Anosov case} \label{Sec4}

In this section we prove Theorem \ref{mainth} for a pseudo-Anosov $\phi$. First let us recall the following classical fact.

\begin{fact}[Five-term exact sequence]\cite[Corollary 6.4]{Brown} \label{fact1}
	Let
	\begin{equation*} \label{SES1}
		\xymatrix{
			1 \ar[r]   & P \ar[r]  & G \ar[r] & Q \ar[r]   & 1. 
		}
	\end{equation*}
be a short exact sequence of groups. Then the following sequence is exact.
	\begin{equation*} \label{SES2}
	\xymatrix{
		\H_2(G, \Q)\ar[r] & \H_2(Q, \Q) \ar[r]   & \H_1(P, \Q)_Q \ar[r]  & \H_1(G, \Q) \ar[r] & \H_1(Q, \Q) \ar[r]   & 0. 
	}
\end{equation*}
\end{fact}
Note that if $P$ is finitely generated, then $\dim \H_1(P, \Q) < \infty$. In this case, since $\Q$ is a field, it follows that we have a canonical isomorphism $\H_1(P, \Q)_Q \cong \H_1(P, \Q)^Q$.
Therefore Fact \ref{fact1} immediately implies the following.
\begin{corollary}\label{corr}
		Let
	\begin{equation*} \label{SES3}
		\xymatrix{
			1 \ar[r]   & P \ar[r]  & G \ar[r] & \Z \ar[r]   & 1. 
		}
	\end{equation*}
	be a short exact sequence of groups. Assume that $P$ is finitely generated. Then the sequence
	\begin{equation*} \label{SES4}
		\xymatrix{
			0 \ar[r]   & \H_1(P, \Q)^\Z \ar[r]  & \H_1(G, \Q) \ar[r] & \Q  \ar[r]   & 0,
		}
	\end{equation*}
is exact. In particular, we have
$$\H_1(G, \Q) \cong \Q \oplus \H_1(P, \Q)^\Z.$$
\end{corollary}

The key tool for our proof is the following deep theorem, which is a consequence of results of Haglund \cite{H8}, Haglund-Wise \cite{HW8}, and Agol \cite{Agol8, Agol13}.

\begin{theorem} \cite[Corollary 4.2.3 (3)]{Friedl} \label{virt}
	Let $M$ be a closed hyperbolic $3$-manifold. Then $vb_1(M) = \infty$, i.e. 
	there exist finite regular covers $\widetilde{M} \to M$ with arbitrary large $\dim \H_1(\widetilde{M}, \Q)$.
\end{theorem}

Now we can prove Theorem \ref{mainth} for the pseudo-Anosov case.

\begin{prop} \label{anosov}
	Let $\phi \in \Mod(\S)$ be a pseudo-Anosov element and let $N > 0$ be an integer. Then there exists a regular finite cover $\widetilde{\S} \to \S$ with $\dim \H_1^{f.o.}(\widetilde{\S}, \Q)^\phi \geq N$.
\end{prop}

\begin{proof}
	Let $M_\phi$ be the mapping torus of $\phi$. 
	Since there is a fiber bundle $\S \to M_\phi \to S^1$ with aspherical $S^1$ and $\S$, we obtain the short exact sequence of groups
	\begin{equation} \label{SES}
		\xymatrix{
			1 \ar[r]   & \pi_1(\S) \ar[r]  & \pi_1(M_\phi)  \ar[r] & \Z  \ar[r]   & 1. 
		}
	\end{equation}
	
	By Theorem \ref{hyp} the $3$-manifold $M_\phi$ admits a hyperbolic metric. Theorem \ref{virt} implies that there exists a regular cover $q: \widetilde{M} \to M_\phi$ such that $b_1(\widetilde{M}) = \dim \H_1(\widetilde{M}, \Q) \geq N+1$. We have the normal subgroup $\pi_1(\widetilde{M}) \lhd \pi_1(M_\phi)$ is of finite index. Denote $\Gamma = \pi_1(\widetilde{M}) \cap \pi_1(\S)$. Since $ \pi_1(\widetilde{M})$ and  $\pi_1(\S)$ are normal subgroups in $\pi_1(M_\phi)$, then so is $\Gamma$.
	
	Since $[\pi_1(M_\phi) : \pi_1(\widetilde{M})] < \infty$, then $[\pi_1(\S) : \Gamma] < \infty$. Therefore, there exists a finite regular cover $p: \widetilde{\S} \to \S$, such that $p_*(\pi_1(\widetilde{\S})) = \Gamma$. Therefore, (\ref{SES}) implies that we obtain the following commutative diagram
	\begin{equation} \label{cover}
		\xymatrix{
			1 \ar[r]   & \pi_1(\widetilde{\S}) \ar[r] \ar@{^{(}->}[d]^{p_*}  & \pi_1(\widetilde{M})  \ar[r] \ar@{^{(}->}[d]^{q_*}  & \Z  \ar[r] \ar@{^{(}->}[d]^{n}  & 1 \\
			1 \ar[r]   & 	\pi_1(\S)  \ar[r]  & \pi_1(M_\phi)  \ar[r]   & \Z  \ar[r]  & 1 
		}
	\end{equation}
	where the right vertical arrow is the multiplication by some integer $n$.  
	Since $[\pi_1(M_\phi) : \pi_1(\widetilde{M})] < \infty$, it follows that $p_*(\pi_1(\widetilde{\S})) = \Gamma \neq \pi_1(\widetilde{M})$. Therefore, we obtain that $n \neq 0$.
	 
	Note that the group $\Z$ in the bottom line of (\ref{cover}) acts on $\pi_1(\S)$ by $\phi$. Hence the group $\Z$ in the top line of (\ref{cover}) acts on $\pi_1(\widetilde{\S})$ by $\phi^n$. In particular, $\pi_1(\widetilde{\S})$ is $\phi^n$-invariant. 
	Then Corollary \ref{corr} applied to the top line of (\ref{cover}) implies
	$$
		\H_1(\pi_1(\widetilde{M}), \Q) \cong \Q \oplus \H_1(\widetilde{\S}, \Q)^{\phi^n},
	$$
	and since $\widetilde{M}$ is also hyperbolic, we obtain $\H_1(\widetilde{M}, \Q) \cong \H_1(\pi_1(\widetilde{M}), \Q)$.
	
	Hence we have 
	$$\dim \H_1(\widetilde{\S}, \Q)^{\phi^n} =  \dim \H_1(\pi_1(\widetilde{M}), \Q) - 1 = \dim \H_1(\widetilde{M}, \Q)  - 1 = b_1(\widetilde{M}) - 1 \geq N.$$
	Since 
	$$ \H_1(\widetilde{\S}, \Q)^{\phi^n} \subseteq \H_1^{f.o.}(\widetilde{\S}, \Q)^\phi,$$
	it follows that 
	$$ \dim \H_1^{f.o.}(\widetilde{\S}, \Q)^\phi \geq \dim  \H_1(\widetilde{\S}, \Q)^{\phi^n} \geq N.$$
	This concludes the proof of Theorem \ref{mainth}.
\end{proof}

Now we can finish the proof of the main result.

\begin{proof}[Proof of Theorem \ref{mainth}]
	By the Nielsen-Thurston classification we have that $\phi$ is periodic, reducible, or pseudo-Anosov. If $\phi$ is periodic, then $\phi^n = \id$ for some integer $n > 0$. In this case for any characteristic cover $\widetilde{\S} \to \S$ we obtain $$\H_1^{f.o.}(\widetilde{\S}, \Q)^\phi = \H_1^{f.o.}(\widetilde{\S}, \Q)^{\phi^n} = \H_1^{f.o.}(\widetilde{\S}, \Q)^{\id} = \H_1(\widetilde{\S}, \Q).$$
	Hence the result follows from the fact that there exist finite characteristic covers of an arbitrary large degree.
	
	For reducible and pseudo-Anosov cases, by Propositions \ref{red} and \ref{anosov}, we obtain that there exists a finite regular cover $\widetilde{\S} \to \S$ such that $\dim \H_1^{f.o.}(\widetilde{\S}, \Q)^\phi \geq N$. Note that $ \H_1^{f.o.}(\widetilde{\S}, \Q)^\phi \cong  \H_1^{f.o.}(K, \Q)^\phi $, where $K = \pi_1(\widetilde{\S}) \lhd \pi_1(\S)$.
	Corollary \ref{cor1} implies that there exists a finite index characteristic subgroup $K' \lhd \pi_1(\S)$ such that $\dim \H_1^{f.o.}(K', \Q)^\phi \geq \dim \H_1^{f.o.}(K, \Q)^\phi$. Denote by $\widetilde{\S}' \to \S$ the characteristic cover corresponding to $K'$. Then we obtain 
	$$\dim \H_1^{f.o.}(\widetilde{\S}', \Q)^\phi = \dim \H_1^{f.o.}(K', \Q)^\phi  \geq \dim \H_1^{f.o.}(K, \Q)^\phi  = \H_1^{f.o.}(\widetilde{\S}, \Q)^\phi  \geq N.$$
	This concludes the proof.
\end{proof}

\end{document}